\documentclass[11pt]{amsart}

\usepackage[T1]{fontenc}
\usepackage{lmodern}
\usepackage{amsmath,amssymb,amsthm,mathtools}
\usepackage[margin=1in]{geometry}
\usepackage{microtype}
\usepackage{booktabs,array}
\usepackage{tikz}
\usepackage{xcolor}
\usepackage[colorlinks=true,linkcolor=blue!45!black,citecolor=blue!45!black,
  urlcolor=blue!45!black]{hyperref}
\allowdisplaybreaks

\hypersetup{
  pdftitle={The Five-Point Case of Mordell's Discriminant Inequality},
  pdfauthor={Jie Wang}
}

\newtheorem{theorem}{Theorem}[section]
\newtheorem{proposition}[theorem]{Proposition}
\newtheorem{lemma}[theorem]{Lemma}
\newtheorem{corollary}[theorem]{Corollary}
\theoremstyle{definition}

\theoremstyle{remark}
\newtheorem{remark}[theorem]{Remark}

\numberwithin{equation}{section}

\newcommand{\C}{\mathbb C}
\newcommand{\R}{\mathbb R}
\newcommand{\Q}{\mathbb Q}
\newcommand{\ii}{\mathrm i}

\DeclareMathOperator{\Disc}{Disc}
\DeclareMathOperator{\Res}{Res}

\title[The five-point discriminant inequality]
  {The Five-Point Case of Mordell's Discriminant Inequality}
\author{Jie Wang}
\address{State Key Laboratory of Mathematical Sciences, Academy of Mathematics and Systems Science, Chinese Academy of Sciences, Beijing, China}
\email{wangjie212@amss.ac.cn}
\date{}

\begin{document}

\begin{abstract}
Let $z_1,\ldots,z_5\in\C$ satisfy $\sum_{j=1}^5|z_j|^2=5$.
We prove the sharp inequality
\[
 \prod_{1\le i<j\le5}|z_i-z_j|^2\le 5^5,
\]
with equality precisely for a regular pentagon on the unit circle, up to rotation and permutation.
This settles the remaining five-point case of the discriminant problem considered by Mordell in 1960.
The proof first identifies a maximizer with a normalized relative equilibrium of five identical planar point vortices.  We then give an independent exact classification of the corresponding stationary equations.  The main algebraic step proves that every collision-free physical solution of the normalized stationary equations is reflection symmetric; it is obtained from a resultant construction and exact Gr\"obner-basis elimination over $\Q$.  The reflection-symmetric equations have exactly five physical shapes, and an exact discriminant comparison selects the regular pentagon.  All symbolic eliminations, Sturm counts, and interval inequalities are reproduced by an accompanying verification script using rational arithmetic.  As consequences, the same configuration gives the Gaussian weighted Fekete set and the global mode of the five-particle complex Ginibre density after the natural rescaling.
\end{abstract}

\maketitle

\section{Introduction}\label{sec:intro}

The purpose of this paper is to settle the five-point case of a classical discriminant extremal problem.  For an integer $n\ge2$ and $z=(z_1,\ldots,z_n)\in\C^n$, put
\begin{equation}\label{eq:defDelta}
 \Delta_n(z)=\prod_{1\le i<j\le n}|z_i-z_j|^2,
 \qquad
 Q(z)=\sum_{i=1}^n|z_i|^2,
 \qquad
 M_n=\max_{Q(z)=n}\Delta_n(z).
\end{equation}
If $p(X)=\prod_{j=1}^n(X-z_j)$, then
$\Delta_n(z)=|\Disc(p)|$.  Thus $M_n$ is the largest absolute discriminant of a monic polynomial whose roots have prescribed quadratic moment.  The $n$th roots of unity give $M_n\ge n^n$.

Birch conjectured that equality always holds.  Mordell~\cite{Mordell1960} proved the conjecture for $n=3$ and also proved it under the additional hypothesis $|z_1|=\cdots=|z_n|=1$.  The problem is closely related to inequalities of Mulholland~\cite{Mulholland1958}.  Chalk~\cite{Chalk1960} gave the centered regular $(n-1)$-gon as a counterexample for every $n\ge6$.  In the normalization $Q(z)=n$, that configuration has value
\begin{equation}\label{eq:chalkvalue}
 (n-1)^{n-1}\left(\frac{n}{n-1}\right)^{n(n-1)/2},
\end{equation}
which exceeds $n^n$ from $n=6$ onward.  Lin~\cite{Lin2011} subsequently proved the four-point case.  For completeness, Appendix~\ref{app:n4} gives a short self-contained proof of that case based on a Hadamard change of variables.  The five-point problem was therefore the only unresolved case; it was also singled out in the complex polynomial optimization study of Wang and Magron~\cite{WangMagron2025}.

Our main result is the following.

\begin{theorem}\label{thm:main}
Let $z_1,\ldots,z_5\in\C$ satisfy
\[
 \sum_{j=1}^5|z_j|^2=5.
\]
Then
\begin{equation}\label{eq:mainineq}
 \prod_{1\le i<j\le5}|z_i-z_j|^2\le5^5.
\end{equation}
Equality holds if and only if, after a permutation,
\begin{equation}\label{eq:equalitycase}
 z_j=e^{\ii\theta}\omega^{j-1},
 \qquad
 \omega=e^{2\pi\ii/5},
 \qquad j=1,\ldots,5,
\end{equation}
for some $\theta\in\R$.
\end{theorem}

The main issue is not to verify that the regular pentagon is stationary, but to exclude stationary configurations with unequal radii.  The Euler--Lagrange equations are the familiar relative-equilibrium equations for identical point vortices.  Exact classifications of the point-vortex equations by computational algebra are known; in particular, Faug\`ere and Svartz~\cite{FaugereSvartz2012} computed all equilibria for small numbers of identical vortices.  Their method reduces the $N=5$ polynomial system from $2060$ complex solutions to a symmetric system with $17$ solutions.  Related five-vortex configurations and their stability are discussed by Dirksen and Aref~\cite{DirksenAref2011}; see also Aref's survey~\cite{Aref2007}.

The proof here is independent of that classification.  Its key step is a reflection-symmetry theorem for the five-point stationary equations.  We complexify the equations, encode one half of the system in the root polynomial
\[
 p(X)=X^5+aX^3+bX^2+cX+d,
\]
and recover the other half from a resultant.  Three mixed Newton identities give three polynomial equations in $a,b,c,d$.  An exact elimination in the generic chart $ab\ne0$ produces two quartic branches in the invariant $B=b^2$ and, more importantly, the identity
\[
 \delta r_2^2=b^2r_3^3,
\]
where $\delta=\Disc(p)$ and $r_3,r_2$ are coefficients of the resultant polynomial.  The physical conjugacy condition then forces $B$ to be real, and the remaining coefficients become real after at most a quarter-turn.  The exceptional charts are elementary exact eliminations.  It follows that every collision-free physical solution of the normalized stationary equations is reflection symmetric.

Once reflection symmetry is known, the classification becomes small.  There are three possibilities: one real point and two conjugate pairs, three real points and one conjugate pair, or five real points.  The first case leads to the quartic
\begin{equation}\label{eq:introquartic}
 10H^4-65H^3+140H^2-97H+6=0,
\end{equation}
whose two real roots give the two nonclassical five-vortex configurations.  The second case adds no new solution, and the third gives a single collinear configuration.  Altogether there are exactly five physical shapes.  Their discriminants can then be compared exactly; only the regular pentagon reaches $5^5$.

The algebraic computations in the proof are computer-assisted but exact.  They are performed over $\Q$; no floating-point calculation is used to establish an identity, a root count, or an inequality.  The accompanying script \texttt{n5\_complete\_verification.py}\footnote{The script can be downloaded at \url{https://wangjie212.github.io/jiewang/codes/n5_complete_verification.py} and uses SymPy~\cite{Meurer2017}.} reconstructs every polynomial system from the definitions, computes the saturated Gr\"obner bases, checks all normal forms, applies Sturm root counting, and verifies the rational interval bounds. Decimal coordinates appearing below are only illustrations.

The paper is organized as follows.  Section~\ref{sec:stationary} reduces the extremal problem to the normalized point-vortex equations.  Section~\ref{sec:poly} develops the polynomial and resultant encoding.  Section~\ref{sec:reflection} proves reflection symmetry.  Section~\ref{sec:classification} classifies all reflection-symmetric equilibria.  Section~\ref{sec:discriminants} evaluates their discriminants and proves Theorem~\ref{thm:main}.  Section~\ref{sec:connections} records consequences for weighted Fekete points, the complex Ginibre ensemble, and point-vortex relative equilibria.  The first two appendices collect the exact elimination data used in the computer-assisted steps, while Appendix~\ref{app:n4} gives an elementary proof of the four-point case.

\section{From the extremal problem to stationary equations}\label{sec:stationary}

We begin with the variational reduction.  Although we only need $n=5$, it is useful to keep the first observation in general form.

\begin{lemma}\label{lem:maxexists}
The maximum $M_n$ in \eqref{eq:defDelta} exists.  Every maximizer with positive value is collision-free.
\end{lemma}

\begin{proof}
The set $\{z\in\C^n:Q(z)=n\}$ is compact and $\Delta_n$ is continuous, so a maximum exists.  A regular $n$-gon has positive product.  Hence a maximizer cannot contain two equal points.
\end{proof}

For five points set
\[
 L(z)=\log\Delta_5(z)
 =\sum_{1\le i<j\le5}\log|z_i-z_j|^2.
\]
At a collision-free maximizer, Lagrange multipliers may be applied to $L$ subject to $Q(z)=5$.

\begin{proposition}\label{prop:equilibrium}
Every maximizer in Theorem~\ref{thm:main} satisfies
\begin{equation}\label{eq:equilibrium}
 \sum_{j\ne i}\frac1{z_i-z_j}=2\overline{z_i},
 \qquad i=1,\ldots,5.
\end{equation}
In particular,
\begin{equation}\label{eq:centroid}
 \sum_{i=1}^5z_i=0.
\end{equation}
\end{proposition}

\begin{proof}
Wirtinger differentiation gives
\[
 \frac{\partial L}{\partial\overline z_i}
 =\sum_{j\ne i}\frac1{\overline z_i-\overline z_j}.
\]
Thus there is a real multiplier $\lambda$ such that
\[
 \sum_{j\ne i}\frac1{z_i-z_j}=\lambda\overline z_i.
\]
Multiplying by $z_i$ and summing over $i$ yields
\[
 \lambda\sum_i|z_i|^2
 =\sum_{i<j}\left(\frac{z_i}{z_i-z_j}
   +\frac{z_j}{z_j-z_i}\right)=10.
\]
Since $Q(z)=5$, we have $\lambda=2$, proving \eqref{eq:equilibrium}.  Summing \eqref{eq:equilibrium} over $i$ gives $2\sum_i\overline z_i=0$, because the terms on the left cancel in pairs.  This gives \eqref{eq:centroid}.
\end{proof}

Equation~\eqref{eq:equilibrium} is a normalized form of the relative-equilibrium equation for five identical point vortices.  The normalization is immaterial for the shape: if
\[
 \sum_{j\ne i}\frac1{z_i-z_j}=\lambda\overline z_i,
 \qquad \lambda>0,
\]
then a positive real scaling reduces the coefficient to $2$.

There is also an unconstrained formulation that will be useful in Section~\ref{sec:connections}.  Define
\begin{equation}\label{eq:gaussianenergy}
 \Phi(z)=Q(z)-\log\Delta_5(z).
\end{equation}
If $Q(v)=10$, then for $r>0$,
\begin{equation}\label{eq:scaleenergy}
 \Phi(rv)-\Phi(v)
 =10\bigl(r^2-1-\log r^2\bigr)\ge0.
\end{equation}
Thus minimizing $\Phi$ over $\C^5$ is equivalent, after scaling by $\sqrt2$, to maximizing $\Delta_5$ on $Q=5$.  The stationary equation for \eqref{eq:gaussianenergy} is the coefficient-one version of \eqref{eq:equilibrium}.

\section{Polynomial encoding of the stationary equations}\label{sec:poly}

We now study all collision-free solutions of \eqref{eq:equilibrium}.  The argument in Sections~\ref{sec:poly} and~\ref{sec:reflection} does not assume that the solution is a maximizer.

\subsection{Complexification}

Introduce independent variables $x_i,y_i$ and consider
\begin{equation}\label{eq:complexified}
 \sum_{j\ne i}\frac1{x_i-x_j}=2y_i,
 \qquad
 \sum_{j\ne i}\frac1{y_i-y_j}=2x_i,
 \qquad i=1,\ldots,5.
\end{equation}
A physical solution of \eqref{eq:equilibrium} is obtained by setting $x_i=z_i$ and $y_i=\overline z_i$.

The complexified system is invariant under
\begin{equation}\label{eq:Cstar}
 (x_i,y_i)\longmapsto (t x_i,t^{-1}y_i),
 \qquad t\in\C^\times.
\end{equation}
If a physical solution is transformed by \eqref{eq:Cstar}, then
\begin{equation}\label{eq:physicalscaled}
 y_i=\kappa\overline{x_i},
 \qquad
 \kappa=|t|^{-2}>0.
\end{equation}

Summing the equations in \eqref{eq:complexified} and then multiplying the first set by $x_i$ and summing gives
\begin{equation}\label{eq:momentsxy}
 \sum_i x_i=\sum_i y_i=0,
 \qquad
 \sum_i x_i y_i=5.
\end{equation}
Consequently the monic polynomial of the $x_i$ has the form
\begin{equation}\label{eq:pcoeff}
 p(X)=\prod_{i=1}^5(X-x_i)
 =X^5+aX^3+bX^2+cX+d.
\end{equation}
Let
\[
 \delta=\Disc(p)\ne0.
\]
The nonvanishing follows from collision-freeness.

\subsection{The dual root polynomial}

For a simple root $x_i$ of $p$,
\[
 \frac{p''(x_i)}{2p'(x_i)}
 =\sum_{j\ne i}\frac1{x_i-x_j}.
\]
Thus the first half of \eqref{eq:complexified} is equivalent to
\begin{equation}\label{eq:xi}
 \xi_i:=\frac{p''(x_i)}{p'(x_i)}=4y_i.
\end{equation}
Introduce the resultant
\begin{equation}\label{eq:resultant}
 R(Y,T)=\Res_X\bigl(p(X),(Y+TX)p'(X)-p''(X)\bigr).
\end{equation}
Since $p$ is monic of degree five and $\prod_i p'(x_i)=\Disc(p)$,
\begin{equation}\label{eq:Rproduct}
 R(Y,T)=\delta\prod_{i=1}^5(Y+Tx_i-\xi_i).
\end{equation}
Write
\begin{equation}\label{eq:R0coeff}
 R(Y,0)=\delta Y^5+r_3Y^3+r_2Y^2+r_1Y+r_0
\end{equation}
and
\begin{equation}\label{eq:RTcoeff}
 R_T(Y,0)=20\delta Y^3+s_2Y^2+s_1Y+s_0.
\end{equation}
The missing $Y^4$ coefficients follow from $\sum_i x_i=\sum_i\xi_i=0$.  Moreover, by \eqref{eq:momentsxy} and \eqref{eq:xi},
\[
 \sum_i x_i\xi_i=20,
\]
which gives the coefficient $20\delta$ in \eqref{eq:RTcoeff}.

The monic polynomial with roots $y_1,\ldots,y_5$ is obtained directly from \eqref{eq:R0coeff}:
\begin{equation}\label{eq:dualpoly}
 p_y(Y)=\frac{R(4Y,0)}{4^5\delta}
 =Y^5+\frac{r_3}{16\delta}Y^3
 +\frac{r_2}{64\delta}Y^2
 +\frac{r_1}{256\delta}Y
 +\frac{r_0}{1024\delta}.
\end{equation}

The second half of \eqref{eq:complexified}, together with $\xi_i=4y_i$, becomes
\begin{equation}\label{eq:xi-second}
 \sum_{j\ne i}\frac1{\xi_i-\xi_j}=\frac{x_i}{2}.
\end{equation}
The next lemma gives the three relations we shall use.

\begin{lemma}\label{lem:mixedmoments}
Every collision-free solution of \eqref{eq:complexified} satisfies
\begin{equation}\label{eq:mixedmoments}
 \sum_i x_i\xi_i^2=0,
 \qquad
 \sum_i x_i\xi_i^3=7\sum_i\xi_i^2,
 \qquad
 \sum_i x_i\xi_i^4=6\sum_i\xi_i^3.
\end{equation}
Equivalently, the coefficients in \eqref{eq:R0coeff}--\eqref{eq:RTcoeff} satisfy
\begin{equation}\label{eq:threeeqs}
 s_2=0,
 \qquad
 s_1=6r_3,
 \qquad
 s_0=2r_2.
\end{equation}
\end{lemma}

\begin{proof}
Multiply \eqref{eq:xi-second} by $\xi_i^m$ and sum over $i$.  Pairing the terms indexed by $i,j$ gives
\[
 \sum_i\sum_{j\ne i}\frac{\xi_i^m}{\xi_i-\xi_j}
 =\sum_{i<j}\frac{\xi_i^m-\xi_j^m}{\xi_i-\xi_j}.
\]
For $m=2$, this is $(5-1)\sum_i\xi_i=0$.  For $m=3$, it equals
\[
 4\sum_i\xi_i^2+\sum_{i<j}\xi_i\xi_j
 =\frac72\sum_i\xi_i^2,
\]
because $\sum_i\xi_i=0$.  For $m=4$, it equals
\[
 4\sum_i\xi_i^3
 +\sum_{i<j}(\xi_i^2\xi_j+\xi_i\xi_j^2)
 =3\sum_i\xi_i^3.
\]
Comparison with the right-hand side of \eqref{eq:xi-second} gives \eqref{eq:mixedmoments}.  Expanding \eqref{eq:Rproduct} and applying Newton identities converts these three equalities into \eqref{eq:threeeqs}.  The conversion is also reconstructed symbolically in the verification script.
\end{proof}

\section{Reflection symmetry of every five-point equilibrium}\label{sec:reflection}

The goal of this section is to prove the following structural statement.

\begin{theorem}\label{thm:reflection}
Every collision-free physical solution of \eqref{eq:equilibrium} is invariant under reflection in some line through the origin.
\end{theorem}

The proof uses the exact elimination described below.  The detailed input polynomials are recorded in Appendix~\ref{app:generic}.

\subsection{The generic chart}

Assume first that $a\ne0$ and $b\ne0$ in \eqref{eq:pcoeff}.  By the scaling \eqref{eq:Cstar}, we may arrange
\[
 a=1.
\]
Indeed, under $x_i\mapsto tx_i$ the coefficients transform as
\[
 (a,b,c,d)\longmapsto(t^2a,t^3b,t^4c,t^5d).
\]
Set
\begin{equation}\label{eq:Bk}
 B=b^2,
 \qquad
 k=\frac db.
\end{equation}
Substituting $a=1$ and $d=bk$ into \eqref{eq:threeeqs}, dividing the first and third equations by their known factor $b$, and replacing $b^2$ by $B$ gives three polynomials
\[
 F_1(B,c,k),\quad F_2(B,c,k),\quad F_3(B,c,k)\in\Q[B,c,k].
\]
They are displayed in \eqref{eq:F1}--\eqref{eq:F3} in Appendix~\ref{app:generic}.

Let $\widetilde\delta(B,c,k)$ be the discriminant after the same substitution.  Since $b\ne0$ and the roots are distinct, the relevant ideal is the saturation
\begin{equation}\label{eq:Isat}
 \mathcal I
 =\langle F_1,F_2,F_3\rangle:
 (B\widetilde\delta)^\infty.
\end{equation}

\begin{proposition}[Exact generic elimination]\label{prop:genericelim}
The ideal \eqref{eq:Isat} is the intersection of the following two zero-dimensional ideals:
\begin{equation}\label{eq:decomp}
 \mathcal I
 =\langle P_1,c-C_1,k-K_1\rangle
 \cap
 \langle P_2,c-C_2,k-K_2\rangle,
\end{equation}
where
\begin{align}
 P_1(B)&=125B^4+7900B^3-37040B^2+44608B+512,\label{eq:P1}\\
 P_2(B)&=2500B^4-114000B^3+338925B^2-11000B+6912,\label{eq:P2}
\end{align}
and
\begin{align}
 C_1(B)&=\frac{3(125B^3+8020B^2-16496B+4224)}{57088},\label{eq:C1}\\
 K_1(B)&=\frac{75B^3+4812B^2-16320B+13952}{28544},\label{eq:K1}\\
 C_2(B)&=-\frac{1250B^3-69425B^2+327305B-145296}{443535},\label{eq:C2}\\
 K_2(B)&=-\frac{28500B^3-1287200B^2+3884705B-1798816}{5913800}.\label{eq:K2}
\end{align}
The factors $P_1$ and $P_2$ are coprime and $P_1P_2$ is square-free.  In addition, on $\mathcal I$ one has
\begin{equation}\label{eq:dualinvariants}
 \delta r_2^2=b^2r_3^3,
 \qquad
 \delta r_1=cr_3^2,
 \qquad
 \delta r_0=kr_3r_2.
\end{equation}
\end{proposition}

\begin{proof}
All assertions are identities over $\Q$.  The saturation is computed by the Rabinowitsch trick: introduce $s$ and compute a Gr\"obner basis of
\[
 \langle F_1,F_2,F_3,sB\widetilde\delta-1\rangle
 \subset\Q[s,k,c,B].
\]
A change to lexicographic order gives a triangular zero-dimensional basis whose eliminant is the monic form of $P_1P_2$; reduction modulo $P_j$ gives precisely \eqref{eq:C1}--\eqref{eq:K2}.  Exact polynomial division verifies the two inclusions in \eqref{eq:decomp}, the coprimality and square-freeness assertions, and zero normal form for each numerator in \eqref{eq:dualinvariants}.  Appendix~\ref{app:generic} gives the input polynomials, and the accompanying verification script reconstructs the resultant \eqref{eq:resultant} from scratch and performs every division over $\Q$.
\end{proof}

We now impose the physical conjugacy condition.  Let
\[
 p_y(Y)=Y^5+a_yY^3+b_yY^2+c_yY+d_y
\]
be the polynomial in \eqref{eq:dualpoly}.  From \eqref{eq:dualinvariants},
\begin{equation}\label{eq:ratio-invariant}
 \frac{b_y^2}{a_y^3}
 =\frac{\delta r_2^2}{r_3^3}
 =b^2=B.
\end{equation}
Since our scaling satisfies \eqref{eq:physicalscaled} and $a=1$,
\[
 a_y=\kappa^2>0,
 \qquad
 b_y=\kappa^3\overline b.
\]
Therefore
\[
 \frac{b_y^2}{a_y^3}=\overline b^{\,2}=\overline B.
\]
Comparison with \eqref{eq:ratio-invariant} yields
\begin{equation}\label{eq:Breal}
 B=\overline B.
\end{equation}
Thus $B\in\R$.  Proposition~\ref{prop:genericelim} then gives $c,k\in\R$.  If $B>0$, then $b$ and $d=bk$ are real, so $p$ has real coefficients.  If $B<0$, then $b$ and $d$ are purely imaginary; after multiplying all roots by $\ii$, the coefficients become
\[
 a'=-1,
 \qquad b'=-\ii b,
 \qquad c'=c,
 \qquad d'=\ii d,
\]
which are real.  Hence the root set is invariant under conjugation after a rotation.

\subsection{Exceptional charts}

It remains to treat $a=0$ or $b=0$.  A zero coefficient of $p$ gives the corresponding zero coefficient of the physical dual polynomial, because \eqref{eq:physicalscaled} transforms the coefficient of degree $5-j$ by the positive factor $\kappa^j$ followed by conjugation.  Adding these necessary dual conditions to \eqref{eq:threeeqs} and saturating by $\delta$ gives the following exact alternatives.

\begin{proposition}\label{prop:exceptional}
For a collision-free physical solution:
\begin{enumerate}
 \item If $a\ne0$ and $b=0$, then after normalizing $a=1$,
 \begin{equation}\label{eq:exceptional1}
  d=0,
  \qquad
  c=\frac3{20}.
 \end{equation}
 Hence $p(X)=X^5+X^3+\frac3{20}X$, whose roots lie on a line.
 \item The chart $a=0$, $b\ne0$ contains no physical solution.
 \item If $a=b=0$, then $cd=0$.  Thus, up to scaling,
 \[
  p(X)=X^5+cX
  \quad\text{or}\quad
  p(X)=X^5+d,
 \]
 corresponding to the centered square or a regular pentagon.
\end{enumerate}
\end{proposition}

\begin{proof}
Normalize $a=1$ in the first case and $b=1$ in the second.  The saturated lexicographic bases are, respectively,
\[
 \left\{s-\frac{3125}{27},\ d,\ c-\frac3{20}\right\},
 \qquad
 \{1\}.
\]
In the third case, the $s$-free part of the saturated basis is generated by $cd$.  These are exact Gr\"obner-basis computations over $\Q$ and are reproduced in Appendix~\ref{app:generic} and in the verification script.  Finally,
\[
 X^5+X^3+\frac3{20}X
 =X\left(X^4+X^2+\frac3{20}\right),
\]
and both roots of the quadratic equation in $X^2$ are negative real numbers; hence all five roots are collinear after rotation.
\end{proof}

\begin{proof}[Proof of Theorem~\ref{thm:reflection}]
The generic case follows from \eqref{eq:Breal} and the discussion after it.  Proposition~\ref{prop:exceptional} treats all remaining coefficient charts.  A monic polynomial with real coefficients has a root set invariant under complex conjugation.  Undoing the common rotation gives reflection symmetry about a line through the origin.
\end{proof}

\section{Classification of the reflection-symmetric equilibria}\label{sec:classification}

Rotate the reflection axis to the real axis.  Because there are five points, the number of real points is $1$, $3$, or $5$.  We treat these cases separately.

For clarity, Table~\ref{tab:notation} separates the coefficient notation used in
the unrestricted and dual polynomials from the local parameters used below.

\begin{table}[ht]
\centering
\caption{Notation used in the elimination and reflection-symmetric classification.}
\label{tab:notation}
\begin{tabular}{@{}ll@{}}
\toprule
Symbols & Meaning\\
\midrule
$a,b,c,d$ & coefficients of the root polynomial $p$ in \eqref{eq:pcoeff}\\
$a_y,b_y,c_y,d_y$ & coefficients of the dual root polynomial $p_y$\\
$h,\eta,\varepsilon$ & one-real-point parameters in \eqref{eq:ueta}--\eqref{eq:vepsilon}\\
$u,v,\beta,C$ & three-real-point parameters in \eqref{eq:32factor}--\eqref{eq:beta32}\\
\bottomrule
\end{tabular}
\end{table}

\subsection{One real point and two conjugate pairs}

Write
\begin{equation}\label{eq:factor12}
 p(X)=(X-h)q_1(X)q_2(X),
 \qquad
 q_j(X)=X^2+u_jX+v_j,
\end{equation}
where $h,u_j,v_j\in\R$ and each $q_j$ has a nonreal conjugate pair of roots.  The centroid and moment conditions are incorporated by writing
\begin{equation}\label{eq:ueta}
 u_1=\frac{h+\eta}{2},
 \qquad
 u_2=\frac{h-\eta}{2},
\end{equation}
and
\begin{equation}\label{eq:vepsilon}
 v_1=\frac{5-h^2+2\varepsilon}{4},
 \qquad
 v_2=\frac{5-h^2-2\varepsilon}{4}.
\end{equation}
Indeed $h-u_1-u_2=0$ and $h^2+2v_1+2v_2=5$.

For a root $z$ of $q_j$, one has $\overline z=-u_j-z$.  The equilibrium equation is therefore equivalent to
\begin{equation}\label{eq:remainders}
 p''(h)-4hp'(h)=0,
\end{equation}
\begin{equation}\label{eq:remq1}
 p''+4(X+u_1)p'\equiv0\pmod{q_1},
\end{equation}
\begin{equation}\label{eq:remq2}
 p''+4(X+u_2)p'\equiv0\pmod{q_2}.
\end{equation}
These are five polynomial equations in $h,\eta,\varepsilon$.

\begin{proposition}\label{prop:one-real-elim}
A lexicographic Gr\"obner basis of \eqref{eq:remainders}--\eqref{eq:remq2} contains
\begin{equation}\label{eq:etaepsilon-rel}
 2\eta\varepsilon=5h(h^2-1),
\end{equation}
\begin{equation}\label{eq:eta2-rel}
 h\bigl(276\eta^2+2450h^8-14875h^6+28975h^4
 -16475h^2-1455\bigr)=0,
\end{equation}
and
\begin{equation}\label{eq:elim-h}
 h(h^2-1)
 \bigl(10h^8-65h^6+140h^4-97h^2+6\bigr)=0.
\end{equation}
The complete six-element basis is given in Appendix~\ref{app:symmetric}.
\end{proposition}

\begin{proof}
Expand \eqref{eq:remainders} and take the two coefficients of each of the linear remainders in \eqref{eq:remq1}--\eqref{eq:remq2}.  Exact Buchberger reduction over $\Q$ gives the basis displayed in Appendix~\ref{app:symmetric}; the last, second, and fifth elements yield \eqref{eq:etaepsilon-rel}--\eqref{eq:elim-h}.  The verification script constructs these five equations directly from \eqref{eq:factor12}--\eqref{eq:vepsilon} and checks the basis.
\end{proof}

We now read off the physical solutions.

If $h=0$, the specialized equations reduce to
\begin{equation}\label{eq:hzero}
 \eta^2+4\varepsilon^2=10,
 \qquad \eta\varepsilon=0,
 \qquad \eta(\eta^2-10)=0.
\end{equation}
Thus either $\eta^2=10$, $\varepsilon=0$, giving the centered square, or $\eta=0$, $\varepsilon^2=5/2$, giving the collinear configuration.

If $h^2=1$, equations \eqref{eq:etaepsilon-rel}--\eqref{eq:eta2-rel} give
\[
 \eta^2=5,
 \qquad \varepsilon=0.
\]
Then $v_1=v_2=1$, and \eqref{eq:factor12} is $X^5-1$ or $X^5+1$ after a possible rotation.  This is the regular pentagon.

It remains to consider
\begin{equation}\label{eq:QH}
 Q_*(H):=10H^4-65H^3+140H^2-97H+6=0,
 \qquad H=h^2.
\end{equation}
On this branch the original equations imply
\begin{equation}\label{eq:JH}
 \eta^2=J(H):=
 \frac{-350H^3+1775H^2-2430H+975}{92},
\end{equation}
and
\begin{equation}\label{eq:epsilonH}
 \varepsilon=\frac{5h(H-1)}{2\eta}.
\end{equation}
These relations are sufficient as well as necessary: after substituting \eqref{eq:JH}--\eqref{eq:epsilonH}, each numerator of the original five equations \eqref{eq:remainders}--\eqref{eq:remq2} has zero remainder modulo
\[
 Q_*(H),\qquad h^2-H,\qquad \eta^2-J(H).
\]

Sturm's theorem gives exactly two real roots of $Q_*$, both positive.  We denote them by $H_-<H_+$; exact isolating intervals are
\begin{equation}\label{eq:Hintervals}
 \frac{68395128919}{10^{12}}<H_-<\frac{68395128920}{10^{12}},
\end{equation}
\begin{equation}\label{eq:Hintervals2}
 \frac{1214112777579}{10^{12}}<H_+<\frac{1214112777580}{10^{12}}.
\end{equation}
On each interval, exact rational interval arithmetic gives $J(H)>0$ and
\begin{equation}\label{eq:discquadneg}
 \frac{5H+J(H)-20}{2}<0.
\end{equation}
The left-hand side of \eqref{eq:discquadneg} is the sum of the two quadratic
discriminants.  Their product is the expression in \eqref{eq:factorA}; exact
rational interval evaluation shows that it is positive on each interval
\eqref{eq:Hintervals}--\eqref{eq:Hintervals2}.  Thus both quadratic
discriminants are negative.  Moreover, $J(H)>0$ gives $\eta\ne0$, so
$q_1\ne q_2$.  Two distinct real quadratics with nonreal conjugate roots cannot
share a root, and the real root $h$ cannot coincide with a root of either
quadratic.  Hence both roots $H_\pm$ give collision-free configurations with
one real point and two distinct nonreal conjugate pairs.  Up to rotation and
relabeling, each $H_\pm$ gives one shape.  Indeed $\eta\mapsto-\eta$,
$\varepsilon\mapsto-\varepsilon$ merely exchanges $q_1$ and $q_2$, while
$h\mapsto-h$, $\eta\mapsto-\eta$ is multiplication of all roots by $-1$.

For reference, the two configurations have the approximate coordinates
\begin{align}\label{eq:Eminus}
 E_-:\quad
 &0.2615246239,\nonumber\\
 &-0.8103785650\pm0.6884696495\ii,\nonumber\\
 &\phantom{-}0.6796162530\pm0.9344625768\ii,
\end{align}
and
\begin{align}\label{eq:Eplus}
 E_+:\quad
 &1.1018678585,\nonumber\\
 &-0.3757025983\pm1.2413092565\ii,\nonumber\\
 &-0.1752313310\pm0.4245426706\ii.
\end{align}
The decimals play no role in the proof.

\subsection{Three real points and one conjugate pair}

Write
\begin{equation}\label{eq:32factor}
 p(X)=c_3(X)q(X),
\end{equation}
where
\[
 c_3(X)=X^3-uX^2+\beta X+C,
 \qquad
 q(X)=X^2+uX+v.
\]
The centroid is automatically zero, and the moment condition gives
\begin{equation}\label{eq:beta32}
 \beta=\frac{u^2+2v-5}{2}.
\end{equation}
The three real roots satisfy $p''-4Xp'=0$, whereas for a root $z$ of $q$ one has $\overline z=-u-z$.  Hence
\begin{equation}\label{eq:32remainders}
 p''-4Xp'\equiv0\pmod{c_3},
 \qquad
 p''+4(X+u)p'\equiv0\pmod q.
\end{equation}

\begin{proposition}\label{prop:32}
The equations \eqref{eq:32remainders} imply
\begin{equation}\label{eq:u32}
 u\bigl(10u^8-55u^6+115u^4-37u^2-9\bigr)=0.
\end{equation}
If $u=0$, they give $C=0$ and $v=5/4$, hence the centered square.  If $u\ne0$ and $U=u^2$, then
\begin{equation}\label{eq:Q32}
 10U^4-55U^3+115U^2-37U-9=0
\end{equation}
and
\begin{equation}\label{eq:V32}
 v=-\frac{130U^3-625U^2+1074U-211}{304}.
\end{equation}
The polynomial in \eqref{eq:Q32} has exactly one positive root, lying in
\[
 0.607600<U<0.607601,
\]
and \eqref{eq:V32} is negative throughout this interval.  Consequently the $u\ne0$ branch cannot contain a nonreal conjugate pair.
\end{proposition}

\begin{proof}
The exact lexicographic basis is displayed in Appendix~\ref{app:symmetric}.  Its last element is \eqref{eq:u32}; the remaining elements give \eqref{eq:V32}.  Sturm's theorem proves the asserted positive-root count, and direct rational interval evaluation proves $v<0$.  If $q$ had a nonreal conjugate pair, then $v$ would be the squared modulus of either root and hence positive, a contradiction.
\end{proof}

\subsection{Five real points}

If all five roots are real, then \eqref{eq:equilibrium} gives
\[
 p''(z_i)=4z_i p'(z_i)
\]
at each root.  Therefore
\[
 p''(X)-4Xp'(X)+20p(X)
\]
has degree at most four and vanishes at five distinct points.  It is identically zero.

\begin{proposition}\label{prop:collinear}
The unique monic degree-five solution of
\begin{equation}\label{eq:HermiteODE}
 p''-4Xp'+20p=0
\end{equation}
with zero $X^4$ coefficient is
\begin{equation}\label{eq:collinearpoly}
 p(X)=X^5-\frac52X^3+\frac{15}{16}X.
\end{equation}
Its roots are
\begin{equation}\label{eq:collinearroots}
 0,
 \qquad
 \pm\sqrt{\frac{5+\sqrt{10}}4},
 \qquad
 \pm\sqrt{\frac{5-\sqrt{10}}4}.
\end{equation}
\end{proposition}

\begin{proof}
Substitution of a general monic quintic with no $X^4$ term into \eqref{eq:HermiteODE} determines the remaining coefficients uniquely and gives \eqref{eq:collinearpoly}.  Factoring the even quartic after removing the root $0$ gives \eqref{eq:collinearroots}.
\end{proof}

Combining Theorem~\ref{thm:reflection} with the three cases above gives the complete stationary classification; see
also Figure~\ref{fig:configurations}.

\begin{theorem}\label{thm:classification}
Up to rotation and permutation, the collision-free physical solutions of \eqref{eq:equilibrium} consist of exactly five configurations:
\begin{enumerate}
 \item the regular pentagon;
 \item the centered square;
 \item the collinear configuration \eqref{eq:collinearroots};
 \item the two nonclassical configurations $E_-$ and $E_+$ determined exactly by \eqref{eq:QH}--\eqref{eq:epsilonH} with $H=H_-$ and $H=H_+$.
\end{enumerate}
\end{theorem}

\begin{figure}[ht]
\centering
\begin{tikzpicture}[x=0.82cm,y=0.82cm]
\begin{scope}[xshift=0.0cm]
\draw[black!20,line width=0.25pt] (-1.65,0)--(1.65,0);
\draw[black!20,line width=0.25pt] (0,-1.65)--(0,1.65);
\fill (1.000000000,0.000000000) circle[radius=1.65pt];
\fill (0.309016994,0.951056516) circle[radius=1.65pt];
\fill (-0.809016994,0.587785252) circle[radius=1.65pt];
\fill (-0.809016994,-0.587785252) circle[radius=1.65pt];
\fill (0.309016994,-0.951056516) circle[radius=1.65pt];
\node[font=\footnotesize] at (0,-2) {Pentagon};
\end{scope}
\begin{scope}[xshift=3.198cm]
\draw[black!20,line width=0.25pt] (-1.65,0)--(1.65,0);
\draw[black!20,line width=0.25pt] (0,-1.65)--(0,1.65);
\fill (0.000000000,0.000000000) circle[radius=1.65pt];
\fill (1.118033989,0.000000000) circle[radius=1.65pt];
\fill (0.000000000,1.118033989) circle[radius=1.65pt];
\fill (-1.118033989,0.000000000) circle[radius=1.65pt];
\fill (-0.000000000,-1.118033989) circle[radius=1.65pt];
\node[font=\footnotesize] at (0,-2) {Centered square};
\end{scope}
\begin{scope}[xshift=6.396cm]
\draw[black!20,line width=0.25pt] (-1.65,0)--(1.65,0);
\draw[black!20,line width=0.25pt] (0,-1.65)--(0,1.65);
\fill (0.000000000,0.000000000) circle[radius=1.65pt];
\fill (-0.677813090,0.000000000) circle[radius=1.65pt];
\fill (0.677813090,0.000000000) circle[radius=1.65pt];
\fill (-1.428485007,0.000000000) circle[radius=1.65pt];
\fill (1.428485007,0.000000000) circle[radius=1.65pt];
\node[font=\footnotesize] at (0,-2) {Collinear};
\end{scope}
\begin{scope}[xshift=9.594cm]
\draw[black!20,line width=0.25pt] (-1.65,0)--(1.65,0);
\draw[black!20,line width=0.25pt] (0,-1.65)--(0,1.65);
\fill (0.261524624,0.000000000) circle[radius=1.65pt];
\fill (-0.810378565,0.688469649) circle[radius=1.65pt];
\fill (-0.810378565,-0.688469649) circle[radius=1.65pt];
\fill (0.679616253,0.934462577) circle[radius=1.65pt];
\fill (0.679616253,-0.934462577) circle[radius=1.65pt];
\node[font=\footnotesize] at (0,-2) {$E_-$};
\end{scope}
\begin{scope}[xshift=12.792cm]
\draw[black!20,line width=0.25pt] (-1.65,0)--(1.65,0);
\draw[black!20,line width=0.25pt] (0,-1.65)--(0,1.65);
\fill (1.101867858,0.000000000) circle[radius=1.65pt];
\fill (-0.375702598,1.241309257) circle[radius=1.65pt];
\fill (-0.375702598,-1.241309257) circle[radius=1.65pt];
\fill (-0.175231331,0.424542671) circle[radius=1.65pt];
\fill (-0.175231331,-0.424542671) circle[radius=1.65pt];
\node[font=\footnotesize] at (0,-2) {$E_+$};
\end{scope}
\end{tikzpicture}
\caption{The five stationary shapes. All panels use the same scale.}
\label{fig:configurations}
\end{figure}
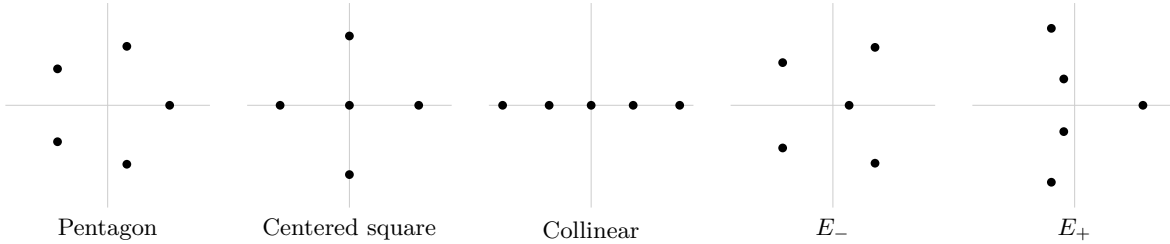

\section{Discriminants and proof of the main theorem}\label{sec:discriminants}

We now evaluate $\Delta_5$ on the five stationary shapes.  For a monic polynomial $p$ with distinct roots,
\begin{equation}\label{eq:discprod}
 |\Disc(p)|=\prod_{i=1}^5|p'(z_i)|
 =\prod_{i<j}|z_i-z_j|^2.
\end{equation}

For the regular pentagon take $p(X)=X^5-1$.  Then $|p'(z_i)|=5$ for every root, so
\begin{equation}\label{eq:pentdisc}
 \Delta_{\mathrm{pent}}=5^5=3125.
\end{equation}

For the centered square, the moment condition gives outer radius $r^2=5/4$, hence
\[
 p(X)=X(X^4-r^4)=X^5-\frac{25}{16}X.
\]
Therefore
\begin{equation}\label{eq:squaredisc}
 \Delta_{\mathrm{sq}}
 =\frac{5^{10}}{2^{12}}
 =2384.185791015625.
\end{equation}

For the collinear polynomial \eqref{eq:collinearpoly}, direct exact evaluation gives
\begin{equation}\label{eq:colldisc}
 \Delta_{\mathrm{col}}
 =\frac{84375}{1024}
 =82.3974609375.
\end{equation}

It remains to evaluate the two nonclassical configurations.  For $p=(X-h)q_1q_2$, the discriminant product formula gives
\begin{equation}\label{eq:discfactor}
 \Disc(p)
 =\Disc(q_1)\Disc(q_2)
 \bigl(q_1(h)q_2(h)\bigr)^2
 \Res(q_1,q_2)^2.
\end{equation}
On the branch $Q_*(H)=0$, substitution of \eqref{eq:JH}--\eqref{eq:epsilonH} and reduction modulo $Q_*$ gives
\begin{align}
 \Disc(q_1)\Disc(q_2)
 &=-\frac{15}{5888}
 \bigl(610H^3-2305H^2+2750H-2777\bigr),\label{eq:factorA}\\
 q_1(h)q_2(h)
 &=\frac5{368}
 \bigl(70H^3-355H^2+670H+81\bigr),\label{eq:factorB}\\
 \Res(q_1,q_2)
 &=-\frac5{368}
 \bigl(130H^3-725H^2+1310H-901\bigr).\label{eq:factorC}
\end{align}
Their product reduces to
\begin{equation}\label{eq:DH}
 D(H)=\frac{84375}{385875968}
 \bigl(
 12854393-32980990H+30197665H^2-7632290H^3
 \bigr).
\end{equation}
On each isolating interval \eqref{eq:Hintervals}--\eqref{eq:Hintervals2}, the factors in \eqref{eq:factorA}--\eqref{eq:factorC} have the signs required by \eqref{eq:discfactor}, and exact interval arithmetic gives
\begin{equation}\label{eq:Dintervals}
 2347<D(H_-)<2349,
 \qquad
 801<D(H_+)<803.
\end{equation}
In particular both values are strictly below $3125$.  For orientation,
\[
 D(H_-)=2347.838167527395\ldots,
 \qquad
 D(H_+)=801.559584831220\ldots.
\]

We summarize the exact comparison in
Table~\ref{tab:values}.

\begin{table}[ht]
\centering
\caption{The stationary shapes at quadratic moment five. Decimal values
are for orientation only; the comparison uses the exact bounds.}
\label{tab:values}
\begin{tabular}{@{}lcc@{}}
\toprule
Configuration & Discriminant absolute value & Approximate value\\
\midrule
Regular pentagon & $5^5$ & $3125$\\
Centered square & $5^{10}/2^{12}$ & $2384.185791$\\
Collinear & $84375/1024$ & $82.397461$\\
$E_-$ & $D(H_-)$ & $2347.838168$\\
$E_+$ & $D(H_+)$ & $801.559585$\\
\bottomrule
\end{tabular}
\end{table}

\begin{proof}[Proof of Theorem~\ref{thm:main}]
By Lemma~\ref{lem:maxexists}, a global maximizer exists and is collision-free.  Proposition~\ref{prop:equilibrium} shows that it satisfies \eqref{eq:equilibrium}.  Theorem~\ref{thm:classification} therefore restricts it to the five configurations listed above.  Equations \eqref{eq:pentdisc}--\eqref{eq:Dintervals} show that the regular pentagon is the unique maximizing shape.  Its moment is $5$, so its circumradius is $1$.  This gives \eqref{eq:equalitycase} and completes the proof.
\end{proof}

\section{Connections and consequences}\label{sec:connections}

The proof identifies the same five-point geometry in three standard settings.  We record the precise statements.

\subsection{Gaussian weighted Fekete points}

Weighted Fekete problems maximize a Vandermonde product multiplied by an external weight; see Saff and Totik~\cite{SaffTotik1997}.  Consider the Gaussian functional
\begin{equation}\label{eq:weightedV}
 \mathcal G(z)
 =e^{-Q(z)}\Delta_5(z),
 \qquad z\in\C^5.
\end{equation}
With the convention that the squared weighted Vandermonde is
\[
 \Delta_5(z)\prod_{j=1}^5w(z_j)^{2(5-1)},
\]
this corresponds to the Gaussian weight
\[
 w(z)=e^{-|z|^2/8}.
\]

\begin{corollary}\label{cor:fekete}
The global maximizers of \eqref{eq:weightedV} are, up to permutation and common rotation,
\begin{equation}\label{eq:weightedpent}
 z_j=\sqrt2\,e^{2\pi\ii(j-1)/5},
 \qquad j=1,\ldots,5.
\end{equation}
The maximum value is
\[
 2^{10}5^5e^{-10}.
\]
\end{corollary}

\begin{proof}
For a nonzero configuration $z$ and $r>0$,
\[
 \mathcal G(rz)=r^{20}e^{-r^2Q(z)}\Delta_5(z).
\]
Along each ray this is maximized at $r^2Q(z)=10$.  Rescale the resulting configuration by $1/\sqrt2$; it then has moment $5$, and Theorem~\ref{thm:main} applies.  Conversely \eqref{eq:weightedpent} attains the stated value.
\end{proof}

\subsection{The complex Ginibre ensemble}

For the standard complex Ginibre ensemble, the joint density of the five eigenvalues is, up to normalization,
\begin{equation}\label{eq:ginibre-density}
 \exp\left(-\sum_{j=1}^5|z_j|^2\right)
 \prod_{i<j}|z_i-z_j|^2;
\end{equation}
see Ginibre~\cite{Ginibre1965}.  Thus \eqref{eq:ginibre-density} is exactly the functional \eqref{eq:weightedV}.

\begin{corollary}\label{cor:ginibre}
The symmetric joint density \eqref{eq:ginibre-density} on the labeled space
$\C^5$ is maximized precisely by the labeled regular pentagons of radius
$\sqrt2$ centered at the origin.  Equivalently, its pushforward to the
unordered configuration space $\C^5/S_5$, with the induced quotient measure,
has the regular-pentagon orbit as its unique global mode up to common rotation.
\end{corollary}

This finite-particle statement is distinct from the usual large-$N$ circular-law description: it gives the exact most likely configuration at $N=5$.

\subsection{Five identical point vortices}

After translating the center of vorticity to the origin and choosing units and a
rotating frame, relative equilibria of identical planar point vortices satisfy
\begin{equation}\label{eq:vortex-general}
 \sum_{j\ne i}\frac1{z_i-z_j}=\lambda\overline z_i,
 \qquad \lambda>0;
\end{equation}
see Aref~\cite{Aref2007}.  Positive scaling changes $\lambda$ but not the shape.  Theorem~\ref{thm:classification} therefore gives an independent exact classification for five identical vortices.

\begin{corollary}\label{cor:vortex}
After translating the center of vorticity to the origin, and modulo positive
scaling, rotation, reflection, and permutation, five distinct identical
vortices have exactly five relative-equilibrium shapes: the regular pentagon,
the centered square, the collinear configuration, and the two nonclassical
configurations determined by the two real roots of \eqref{eq:QH}.
\end{corollary}

The existence of five physical classes is consistent with the exact computational classification of Faug\`ere and Svartz~\cite{FaugereSvartz2012}.  The present proof isolates a comparatively small invariant calculation specialized to $N=5$ and, in addition, proves the reflection symmetry of every collision-free physical solution directly.

\appendix

\section{Exact elimination in the unrestricted chart}\label{app:generic}

For reproducibility we record the three primitive polynomials used in Proposition~\ref{prop:genericelim}.  They are obtained directly from the resultant \eqref{eq:resultant}, the coefficient equations \eqref{eq:threeeqs}, and the substitutions $a=1$, $d=bk$, $B=b^2$:
\begin{align}
F_1={}&27B^2c+3375B^2k^2-135B^2k-2700Bc^2k-144Bc^2\nonumber\\
&+12500Bck^3+1500Bck^2+495Bck+4Bc\nonumber\\
&-5625Bk^3+1125Bk^2-36Bk+320c^4+1600c^3k\nonumber\\
&+224c^3-960c^2k-12c^2+108ck,\label{eq:F1}\\[1ex]
F_2={}&54B^3+2250B^2ck-369B^2c+6250B^2k^3-3375B^2k^2\nonumber\\
&+195B^2k+8B^2-800Bc^3-2800Bc^2k+712Bc^2\nonumber\\
&+4000Bck^2+220Bck-44Bc-900Bk^2+36Bk\nonumber\\
&+1280c^4-512c^3+48c^2,\label{eq:F2}\\[1ex]
F_3={}&-675B^2k+351B^2+180Bc^2+7500Bck^2+15600Bck\nonumber\\
&-2067Bc+37500Bk^3-37125Bk^2+5085Bk+48B\nonumber\\
&-2000c^3k-5440c^3+13800c^2k+1968c^2\nonumber\\
&-5805ck-216c+648k.\label{eq:F3}
\end{align}

The verification proceeds as follows.  First compute
\[
 \widetilde\delta(B,c,k)
 =\Disc\bigl(X^5+X^3+bX^2+cX+bk\bigr)
 \quad\text{with }B=b^2,
\]
which is even in $b$.  Compute a Gr\"obner basis over $\Q$ of
\begin{equation}\label{eq:rabinowitsch}
 \langle F_1,F_2,F_3,sB\widetilde\delta-1\rangle
 \subset\Q[s,k,c,B]
\end{equation}
with graded reverse lexicographic order, and convert it to lexicographic order by FGLM.  The resulting zero-dimensional basis has four elements.  Its $B$-eliminant is the monic scalar multiple of $P_1P_2$ from \eqref{eq:P1}--\eqref{eq:P2}; the other two $s$-free elements express $c$ and $k$ as rational polynomials in $B$.  Reduction modulo $P_1$ and $P_2$ gives \eqref{eq:C1}--\eqref{eq:K2}.  The accompanying script asserts all these equalities by exact polynomial division.

For the dual-invariant identities \eqref{eq:dualinvariants}, form the numerators after substituting $d=bk$ and convert their even powers of $b$ to $B$.  Each has zero normal form modulo the same saturated basis.  No numerical root approximation enters this computation.

For the three exceptional charts, the exact saturated lexicographic bases are
\begin{align*}
a=1,\ b=0:
&\quad
\left\{s-\frac{3125}{27},\ d,\ c-\frac3{20}\right\},\\
a=0,\ b=1:
&\quad
\{1\}.
\end{align*}
For $a=b=0$, the $s$-free elimination ideal is generated by $cd$.

\section{Exact elimination in the reflection-symmetric cases}\label{app:symmetric}

For the one-real-point parametrization \eqref{eq:factor12}--\eqref{eq:vepsilon}, the lexicographic Gr\"obner basis in the variable order $\varepsilon,\eta,h$ is
\begin{align*}
&\frac1{2208}\bigl(552\eta^2+2208\varepsilon^2-2450h^{10}+14875h^8-28975h^6
 +12335h^4+6975h^2-5520\bigr),\\
&\frac12\bigl(2\eta\varepsilon-5h^3+5h\bigr),\\
&-\frac{h}{138}\bigl(780\eta h^9-4900\eta h^7+10025\eta h^5-6251\eta h^3+346\eta h-138\varepsilon\bigr),\\
&\frac{\eta}{552}\bigl(552\eta^2+4350h^{10}-27725h^8+59425h^6-45145h^4
 +11855h^2-5520\bigr),\\
&\frac h{276}\bigl(276\eta^2+2450h^8-14875h^6+28975h^4-16475h^2-1455\bigr),\\
&\frac1{10}h(h-1)(h+1)
 \bigl(10h^8-65h^6+140h^4-97h^2+6\bigr).
\end{align*}
This basis yields \eqref{eq:etaepsilon-rel}--\eqref{eq:elim-h}.  On the nonclassical branch, substituting \eqref{eq:JH}--\eqref{eq:epsilonH} in the original five remainder equations and reducing their numerators modulo
\[
 \langle \eta^2-J(H),\ h^2-H,\ Q_*(H)\rangle
\]
gives zero in every case.

For the three-real-point parametrization \eqref{eq:32factor}--\eqref{eq:beta32}, the lexicographic basis in $C,v,u$ is
\begin{align*}
&-\frac1{228}\bigl(-228C+130u^7-625u^5+1036u^3+359u\bigr),\\
&\frac1{2736}\bigl(1690u^8-8125u^6+13810u^4+3413u^2+2736v-3420\bigr),\\
&\frac u{10}\bigl(10u^8-55u^6+115u^4-37u^2-9\bigr).
\end{align*}
The formula \eqref{eq:V32} has zero normal form after adjoining the nonzero branch equation.

The Sturm counts and all sign checks use rational endpoints.  In particular, the intervals \eqref{eq:Hintervals}--\eqref{eq:Hintervals2} each contain exactly one root of $Q_*$, while $Q_*$ has exactly two real roots in total.  Horner interval evaluation over $\Q$ proves $J(H)>0$, the negativity condition \eqref{eq:discquadneg}, positivity of the three factors in \eqref{eq:discfactor} after their prescribed signs are accounted for, and the bounds \eqref{eq:Dintervals}.  Likewise, the polynomial in \eqref{eq:Q32} has exactly one positive root in $(0.607600,0.607601)$ and \eqref{eq:V32} is negative on that interval.

\begin{remark}[Reproducibility]\label{rem:repro}
The ancillary Python script supplied with the paper requires Python~3 and SymPy~1.14.  It generates the resultant and all remainder equations from the definitions rather than reading precomputed coefficients.  Assertions are performed in exact rational arithmetic.  Floating-point evaluation is used only at the end to print the illustrative coordinates \eqref{eq:Eminus}--\eqref{eq:Eplus}.  A successful run ends with the message \texttt{ALL EXACT CHECKS PASSED}.
\end{remark}

\section{The four-point case}\label{app:n4}

For completeness, we give a short proof of the four-point case.  The argument is independent of the five-point elimination above.  Its main observation is that a Hadamard change of variables converts the four-point Vandermonde into a three-point Vandermonde in three squared coordinates.

\begin{lemma}\label{lem:threepoint}
For arbitrary $u,v,w\in\C$,
\begin{equation}\label{eq:threepoint}
 |(u-v)(v-w)(w-u)|
 \le \frac14\bigl(|u|+|v|+|w|\bigr)^3.
\end{equation}
If the right-hand side is nonzero, equality holds if and only if, after a permutation,
\begin{equation}\label{eq:threepoint-eq}
 u=0,\qquad v=-w,\qquad |v|=|w|.
\end{equation}
\end{lemma}

\begin{proof}
Regard $u,v,w$ as points $A,B,C$ in the Euclidean plane and put
\[
 a=|v-w|,\qquad b=|w-u|,\qquad c=|u-v|.
\]
Let
\[
 m=\min_{P\in\C}(PA+PB+PC).
\]
Since the origin is an admissible point,
\begin{equation}\label{eq:m-bound}
 m\le |u|+|v|+|w|.
\end{equation}
We prove
\begin{equation}\label{eq:abc-m}
 abc\le \frac{m^3}{4},
\end{equation}
with equality only when $A,B,C$ are collinear, one of them is the midpoint of the other two, and that midpoint is a minimizer in the definition of $m$.

Suppose first that $A,B,C$ are collinear, with $A$ between $B$ and $C$.  Then $a=b+c$ and $m=a$, whence
\[
 abc=a\,bc\le a\frac{(b+c)^2}{4}=\frac{a^3}{4}=\frac{m^3}{4}.
\]
Equality holds exactly when $b=c$.

Now suppose that the triangle is nondegenerate.  If one angle is at least $120^\circ$, say $\angle A\ge120^\circ$, the geometric median is $A$, so $m=b+c$.  Since $a<b+c$,
\[
 abc<(b+c)bc\le \frac{(b+c)^3}{4}=\frac{m^3}{4}.
\]
It remains to consider the case in which all angles are less than $120^\circ$.  Then the geometric median $F$ lies in the interior and
\[
 \angle AFB=\angle BFC=\angle CFA=120^\circ.
\]
Write
\[
 x=FA,\qquad y=FB,\qquad z=FC,
 \qquad p=x+y+z=m,
\]
and set
\[
 q=xy+yz+zx,\qquad r=xyz.
\]
The cosine rule gives
\[
 a^2=y^2+yz+z^2,\qquad
 b^2=z^2+zx+x^2,\qquad
 c^2=x^2+xy+y^2.
\]
A direct expansion yields
\begin{equation}\label{eq:fermat-identity}
 a^2b^2c^2=p^2q^2-q^3-p^3r.
\end{equation}
Put $t=q/p^2$ and $\rho=r/p^3$.  Then $0<t\le1/3$, and \eqref{eq:fermat-identity} becomes
\begin{equation}\label{eq:fermat-normal}
 \frac{a^2b^2c^2}{p^6}=t^2(1-t)-\rho.
\end{equation}
If $0<t\le1/4$, then
\[
 t^2(1-t)-\rho\le t^2(1-t)\le \frac{3}{64}<\frac1{16}.
\]
If $1/4\le t\le1/3$, Schur's inequality $p^3+9r\ge4pq$ gives
\[
 \rho\ge\frac{4t-1}{9}.
\]
Hence
\[
 \frac{a^2b^2c^2}{p^6}
 \le g(t):=t^2(1-t)-\frac{4t-1}{9}.
\]
Moreover
\[
 g'(t)=2t-3t^2-\frac49<0,
\]
because $27t^2-18t+4>0$ for every real $t$.  Therefore
\[
 g(t)\le g(1/4)=\frac{3}{64}<\frac1{16}.
\]
Thus \eqref{eq:abc-m} is strict for every nondegenerate triangle.

Combining \eqref{eq:abc-m} with \eqref{eq:m-bound} proves \eqref{eq:threepoint}.  If equality holds and the product is nonzero, the preceding discussion forces the collinear equality case: one point, say $A$, is the midpoint of the other two.  Equality in \eqref{eq:m-bound} then forces the origin itself to be the geometric median, hence $A=0$.  Thus $u=0$ and $v=-w$ with $|v|=|w|$.  The converse is immediate.
\end{proof}

\begin{theorem}[The four-point case]\label{thm:n4}
If $z_1,z_2,z_3,z_4\in\C$ satisfy
\[
 \sum_{j=1}^4|z_j|^2=4,
\]
then
\begin{equation}\label{eq:n4-main}
 \prod_{1\le i<j\le4}|z_i-z_j|^2\le4^4.
\end{equation}
Equality holds if and only if, after a permutation,
\[
 \{z_1,z_2,z_3,z_4\}
 =\{e^{\ii\theta},\ii e^{\ii\theta},-e^{\ii\theta},-\ii e^{\ii\theta}\}
\]
for some $\theta\in\R$.
\end{theorem}

\begin{proof}
Introduce the Hadamard coordinates
\begin{align*}
 s&=\frac{z_1+z_2+z_3+z_4}{2},&
 x&=\frac{z_1+z_2-z_3-z_4}{2},\\
 y&=\frac{z_1-z_2+z_3-z_4}{2},&
 t&=\frac{z_1-z_2-z_3+z_4}{2}.
\end{align*}
The Hadamard matrix divided by $2$ is unitary, and therefore
\begin{equation}\label{eq:n4-parseval}
 |s|^2+|x|^2+|y|^2+|t|^2
 =\sum_{j=1}^4|z_j|^2=4.
\end{equation}
The inverse transformation is
\begin{align*}
 z_1&=\frac{s+x+y+t}{2},&
 z_2&=\frac{s+x-y-t}{2},\\
 z_3&=\frac{s-x+y-t}{2},&
 z_4&=\frac{s-x-y+t}{2}.
\end{align*}
Consequently,
\[
\begin{array}{lll}
 z_1-z_2=y+t,&z_1-z_3=x+t,&z_1-z_4=x+y,\\
 z_2-z_3=x-y,&z_2-z_4=x-t,&z_3-z_4=y-t,
\end{array}
\]
and hence
\begin{equation}\label{eq:n4-factor}
 \prod_{i<j}(z_i-z_j)
 =(x^2-y^2)(x^2-t^2)(y^2-t^2).
\end{equation}
Apply Lemma~\ref{lem:threepoint} to $u=x^2$, $v=y^2$, and $w=t^2$.  By \eqref{eq:n4-parseval},
\begin{align*}
 \left|\prod_{i<j}(z_i-z_j)\right|
 &\le\frac14\bigl(|x|^2+|y|^2+|t|^2\bigr)^3\\
 &\le\frac14\,4^3=16.
\end{align*}
Squaring gives \eqref{eq:n4-main}.

Suppose equality holds.  Equality in the second inequality above and \eqref{eq:n4-parseval} imply $s=0$.  Equality in Lemma~\ref{lem:threepoint}, after permuting $x,y,t$ if necessary, gives
\[
 x=0,\qquad y^2=-t^2,
 \qquad |y|^2=|t|^2=2.
\]
Thus $y=\pm\ii t$.  If $y=\ii t$, then
\[
 z_2=-z_1,\qquad z_4=-z_3,
 \qquad z_3=\ii z_1,
 \qquad |z_1|=1;
\]
the other sign gives the same square with the opposite orientation.  Hence the four points are the vertices of a unit square centered at the origin.  Conversely such a square has product $4^4$.
\end{proof}

\section*{Acknowledgements}
The author acknowledges the use of GPT-6 Astra to assist with
brainstorming, mathematical development, and manuscript drafting.
The author is solely responsible for the final content, analysis, and
conclusions.

\section*{Funding}
The author was supported by the National Key R\&D Program of China (No.~2023YFA1009401) and the National Natural Science Foundation of China (Nos.~12201618 and 12171324).

\bibliographystyle{amsplain}
\bibliography{references}

\end{document}